\documentclass[11pt]{amsart}
\usepackage{fullpage}
\usepackage{amsmath}
\usepackage{ mathrsfs }
\usepackage{ dsfont }
\usepackage{indentfirst}
\usepackage{mathtools}
\usepackage{amsfonts}
\usepackage{amsthm}
\usepackage{graphicx}
\usepackage{float}
\usepackage[english]{babel}
\usepackage{fullpage}
\usepackage{hyperref}
\usepackage{tikz}

\hypersetup{
	colorlinks,
	linkcolor={red!80!black},
	citecolor={red!80!black},
	urlcolor={blue!80!black}
}

\usepackage{enumerate}
\usepackage{enumitem}

\usepackage{tasks}
\usepackage{color}
\usepackage{theoremref}
\usepackage{relsize}
\usepackage{mathdots}

\usepackage{ amssymb }
\usepackage{ textcomp }

\usepackage{accents}

\newtheorem{theorem}{Theorem}[section]
\newtheorem*{theorem*}{Theorem}

\newtheorem{lemma}[theorem]{Lemma}

\newtheorem{conjecture}[theorem]{Conjecture}
\newtheorem{claim}[theorem]{Claim}
\newtheorem{definition}[theorem]{Definition}

\newtheorem{remark}[theorem]{Remark}

\numberwithin{equation}{section}

\def\w{\omega}

\def\er{\text{er}}
\def\ER{\text{ER}}
\def\CR{\text{CR}}

\def\E{\mathbb{E}}

\def\P{\mathbb{P}}

\usepackage{scalerel}

\newenvironment{claimproof}[1]{\par\noindent\textit{Proof of the claim:}\space#1}{\leavevmode\unskip\penalty9999 \hbox{}\nobreak\hfill\quad\hbox{$\blacksquare$}}

\allowdisplaybreaks

\title{On the off-diagonal unordered Erd\H{o}s--Rado numbers}

\author[Araujo]{Igor Araujo}
\address{Department of Mathematics, University of Illinois at Urbana-Champaign, Urbana, IL 61801} 

\email{\parbox[t]{\linewidth}{\{igoraa2, 
		dadongp2\}@illinois.edu}.} 

\author[Peng]{Dadong Peng}

\begin{document}
	
	\begin{abstract}
		Erd\H{o}s and Rado~\cite{er} 
		introduced the Canonical Ramsey numbers $\er(t)$ as the minimum number $n$ such that every edge-coloring of the ordered complete graph $K_n$ contains either a monochromatic, rainbow, upper lexical, or lower lexical clique of order $t$.
		Richer~\cite{richer} 
		introduced the unordered asymmetric version of the Canonical Ramsey numbers $\CR(s,r)$ as the minimum $n$ such that every edge-coloring of the (unorderd) complete graph $K_n$ contains either a rainbow clique of order $r$, or an orderable clique of order $s$.
		
		We show that $\CR(s,r) = O(r^3/\log r)^{s-2}$, which, up to the multiplicative constant, matches the known lower bound and improves the previously best known bound $\CR(s,r) = O(r^3/\log r)^{s-1}$ by Jiang~\cite{jiang}
		. 
		We also obtain bounds on the further variant $\ER(m,\ell,r)$, defined as the minimum $n$ such that every edge-coloring of the (unorderd) complete graph $K_n$ contains either a monochromatic $K_m$, lexical $K_\ell$, or rainbow $K_r$.
	\end{abstract}
	
	\maketitle
	\section{Introduction}
	
	Monochromatic cliques in edge-colorings of the complete graph is one of the most studied topics in Graph Theory. Ramsey Theory studies, when the number of colors is fixed, how large a complete graph must be to ensure a monochromatic copy of a clique of given order. If we allow the number of colors to be arbitrary, we can still ask how large a complete graph must be to ensure a copy of a clique with some special coloring, which we call \emph{canonical} coloring. Erd\H{o}s and Rado~\cite{er} introduced the so-called \emph{Canonical Ramsey Numbers}, which we denote by $\er(t)$, in the context of vertex-ordered graphs.
	
	\begin{theorem}[Erd\H{o}s, Rado~\cite{er}]\label{thm:er}
		For every integer $t$, there exists a minimum number $n=\er(t)$ such that the following holds. For every edge-coloring $\chi: E(K_n) \to \w$ of the complete graph on (ordered) vertex-set $[n]$, there exists a complete subgraph of order $t$ whose coloring is one of the following: 
		\begin{itemize}
			\item monochromatic;
			\item rainbow;
			\item upper lexical (two edges have the same color if and only if they have the same higher endpoint);
			\item lower lexical (two edges have the same color if and only if they have the same lower endpoint).
		\end{itemize}  
	\end{theorem}
	
	In terms of a quantitative bound for $\er(t)$, Lefmann and R\"odl~\cite{lr} proved that
	$$ 2^{c't^2} \le \er(t) \le 2^{c"t^2 \log(t)} $$
	for some constants $c'$ and $c"$. Some improvements on the value of $c'$ and $c"$ were made since (see, e.g.~\cite{ajmp}) but the logarithmic gap in the exponent remains.
	
	Richer~\cite{richer} introduced an unordered variant, where (ordered) monochromatic, upper lexical and lower lexical cliques are seen as special case of the so-called \emph{orderable} cliques in the unordered setting. These numbers, which we denote by $\CR(o,r)$, were called ``Unordered Canonical Ramsey Numbers.''
	
	\begin{definition}
		A coloring $\chi: E(K_n) \to \w$ is called \emph{orderable} if there is an ordering $v_1, v_2, \dots, v_n$ of the vertices of $K_n$ such that two edges have the same color if they have the same lower end, i.e., $\chi(v_iv_j) = \chi (v_iv_k)$ for every $1\le i<j<k\le n$. 
	\end{definition}
	
	\begin{remark}
		Note that in the above definition we only ask that two edges have the same color if they have the same lower end, compared to the ``if and only if'' requirement for the \emph{lower lexical} coloring considered by Erd\H{o}s and Rado.
	\end{remark}
	
	\begin{remark}
		Note that the above definition would be equivalent if one asks that there exists ordering for which two edges have the same color if they have the same upper end.
	\end{remark}
	
	\begin{definition}[Unordered Canonical Ramsey Numbers~\cite{richer}]
		We define the Unordered Canonical Ramsey Number $\CR(o,r)$ as the minimum $n$ such that every coloring $\chi:E(K_n)\to \w$ contains one of the following:
		\begin{itemize}
			\item orderable $K_o$;
			\item rainbow $K_r$.
		\end{itemize}
	\end{definition}
	
	In this note we are interested in the asymptotic behavior of $\CR(o,r)$ for fixed $o$, when $r$ tends to infinity. A construction from Babai~\cite{babai} and an upper bound from Alon, Lefmann, and R\"odl~\cite{alr} show that 
	$$\CR(3,r)=\Theta\left(\frac{r^3}{\log r}\right).$$
	For general $o\ge 4$, Jiang~\cite{jiang} proved that
	$$\Omega\left(\frac{r^3}{\log r}\right)^{o-2}
	\le \CR(o,r) \le 
	O\left(\frac{r^3}{\log r}\right)^{o-1}.$$
	The main contribution of this paper is to close the gap between the lower and upper bounds.
	
	\begin{theorem}\label{thm:main}
		There is a constant $c>0$ such that for every $o$ and $r$ we have
		$$\CR(o,r) \le \left(\frac{c \cdot r^3}{\log r}\right)^{o-2}.$$
	\end{theorem}
	
	Further, we introduce and discuss the following version of Unordered Canonical Ramsey Numbers, where we identify only the lower lexical and upper lexical cliques as \emph{lexical} cliques.
	
	\begin{definition}
		A coloring $\chi: E(K_n) \to \w$ is called \emph{lexical} if there is an ordering $v_1, v_2, \dots, v_n$ of the vertices of $K_n$ such that two edges have the same color if and only if they have the same lower end. 
	\end{definition}
	
	\begin{definition}
		We define the Unordered Erd\H{os}--Rado Number $\ER(m,\ell,r)$ as the minimum $n$ such that every coloring $\chi:E(K_n)\to \w$ contains one of the following:
		\begin{itemize}
			\item monochromatic $K_m$;
			\item lexical $K_\ell$; 
			\item rainbow $K_r$.
		\end{itemize}
	\end{definition}
	
	The focus of this note is the asymptotic behavior of $\ER(m,\ell,r)$ for fixed $m, \ell$, and $r$ tending to infinity. Some other ranges of parameters have also been studied in the past (see, e.g.,~\cite{AJ}). We now make some observations on how the known results on $\CR(o,r)$ translate to results on $\ER(m,\ell,r)$.
	
	The asymptotic of the Unordered Erd\H{os}--Rado Number $\ER(3,3,r)$ is known up to a multiplicative constant~\cite{alr,babai}
	$$\ER(3,3,r) = \Theta\left( \frac{r^3}{\log r} \right).$$
	
	The proof of the lower bound $\CR(o,r)\ge \Omega\left( r^3/\log r \right)^{o-2}$ in~\cite{jiang} extends the coloring from Babai~\cite{babai} for $o=3$ to arbitrary $o\ge 4$ via a blow-up coloring.
	We highlight that this coloring gives a stronger result, because it not only avoids orderable cliques of given size, but do so by also avoiding monochromatic triangles.
	
	\begin{theorem}[Adapted from~\cite{jiang}]\label{thm:lower}
		There is a constant $c>0$ such that for every integers $\ell \ge 3$ and $r\ge 3$ we have
		$$ \ER(3,\ell,r) \ge \left(\frac{c \cdot r^3}{\log r}\right)^{\ell-2}. $$
	\end{theorem}
	
	We will show that this lower bound is best possible up to the constant factor, and also obtain upper bounds on $\ER(m,\ell,r)$ for other fixed values of $m$ and $\ell$.
	
	\begin{theorem}${}$ \label{thm:main2}
		\begin{enumerate}
			\item There is a constant $c>0$ such that for every integers $m\ge 3$ and $r\ge 3$ we have
			$$\ER(m,3,r) \le c \cdot m \cdot \frac{r^3}{\log r}.$$ \label{ER(m,3,r)}
			
			\item There is a constant $c>0$ such that for every integers $\ell\ge 3$ and $r\ge 3$ we have
			$$\ER(3, \ell,r) \le \left(\frac{c \cdot r^3}{\log r}\right)^{\ell-2}.$$ \label{ER(3,ell,r)}
			
			\item There is a constant $c>0$ such that for every integer $r\ge 3$ we have
			$$ \ER(4,4,r) \le   \frac{c \cdot r^7}{(\log r)^2}.$$ \label{ER(m,4,r)}
		\end{enumerate}
	\end{theorem}
	
	\begin{remark}
		The bounds in~(\ref{ER(m,3,r)}) and~(\ref{ER(3,ell,r)}) are best possible, up to the dependency on $m$ in~(\ref{ER(m,3,r)}), and up to constant $c$ in~(\ref{ER(3,ell,r)}), see Theorem~\ref{thm:lower}. 
	\end{remark}
	
	\begin{remark}
		If $R(m)$ denotes the 2-colored Ramsey number of $m$, then note that $\ER(m,4,3) \ge R(m) > (\sqrt{2}+o(1))^m$. One cannot expect an upper bound of the form $(c m r^3/\log r)^2$ for $\ER(m,4,r)$, in contrast to $\ER(m,3,r) \le cm r^3/\log r$.
	\end{remark}
	
	\begin{remark}
		The proof of~(\ref{ER(m,4,r)}) can be extended to give an upper bound for every $m \ge 4$. However, we do not expect the bound in~(\ref{ER(m,4,r)}) to be best possible, see Conjecture~\ref{conj:ER(4,4,r)}. To simplify the presentation we opt to write the proof only for the $m=4$ case. The best known lower bound is from
		$$ \ER(m,4,r) \ge \ER(3,4,r) \ge \Omega(r^3/\log r)^2. $$
	\end{remark}
	
	The rest of the paper is organized as follows. In Section~\ref{sec:main}, we introduce an auxiliary lemma which will be crucial to most of the proofs, prove the main result, Theorem~\ref{thm:main}, and parts~(\ref{ER(m,3,r)}) and~(\ref{ER(3,ell,r)}) of Theorem~\ref{thm:main2}. In Section~\ref{sec:ER(4,4,r)}, we prove part~(\ref{ER(m,4,r)}) of Theorem~\ref{thm:main2}. Finally, in Section~\ref{sec:lower}, we introduce a coloring from~\cite{jiang} proving Theorem~\ref{thm:lower}.
	
	\textbf{Notation.} We denote by $R(k)$ the 2-colored Ramsey number, so that any 2-coloring of the edges of $K_n$ for $n\ge R(k)$ contains a monochromatic copy of $K_k$. Given a coloring $\chi: E(G) \to \w$ and a vertex $v\in V(G)$, we denote by $N_i(v)$ the neighborhood of vertex $v$ in color $i \in \w$. Further, for $U \subset V(G)$, we write $G[U]$ for the induced subgraph of $G$ on vertex-set $U$, and $e_i(U)$ for the number of edges with color $i$ and both endpoints in $U$. For integer $k$, we denote by $[k]$ the set $\{1,2,\dots, k\}$. Similarly, for integers $a<b$, we denote by $[a,b]$ the interval $\{a, a+1, \dots, b\}$.
	
	\section{Colorings with bounded maximum degree} \label{sec:main}
	
	In the proof of Theorem~\ref{thm:main}, we will make use of the following lemma.
	
	\begin{lemma}[Alon, Jiang, Miller, Pritikin; Theorem 5.6 in~\cite{ajmp}] \label{lem:ajmp}
		There is a constant $c>0$ such that the following holds. For every integers $\Delta$ and $r$, if $n> c \cdot \Delta \cdot \frac{r^3}{\log r}$ and $\chi: E(K_n)\to \w$ is a coloring, where each color-class has maximum degree at most $\Delta$, then $\chi$ contains a rainbow clique of order $r$.
	\end{lemma}
	
	In the proofs that follow we will make repetitively use of the above Lemma~\ref{lem:ajmp}, which motivates the following notation to be used from now on.
	
	\begin{definition}
		A \emph{$\Delta$-good} coloring of a graph $G$ is a coloring $\chi: E(G)\to \w$ for which every color-class has maximum degree at most $\Delta$.
	\end{definition}
	
	Part~(\ref{ER(m,3,r)}) of Theorem~\ref{thm:main2} follows from Lemma~\ref{lem:ajmp}, since a coloring avoiding monochromatic $K_m$ and lexical $K_3$ is $m$-good.
	Note that Part~(\ref{ER(3,ell,r)}) of Theorem~\ref{thm:main2} is a consequence of Theorem~\ref{thm:main} since every orderable $K_\ell$ either is a lexical $K_\ell$ or contains a monochromatic copy of $K_3$. We are now ready to prove our main result, Theorem~\ref{thm:main}. 
	
	\begin{proof}[Proof of Theorem~\ref{thm:main}]
		The proof is by induction on $o \ge 3$. Let $c>0$ be a constant larger than the constant given by Lemma~\ref{lem:ajmp} and such that the base case 
		$$\CR(3,r) \le c \cdot \frac{r^3}{\log r}$$
		holds true. Such a constant exists by~\cite{alr}.
		
		For the inductive step, assume $o \ge 4$ and $n \ge \left(\frac{c \cdot r^3}{\log r}\right)^{o-2}$. Observe that, by Lemma~\ref{lem:ajmp}, if $\chi$ is $\left(\frac{c \cdot r^3}{\log r}\right)^{o-3}$-good then $\chi$ contains a rainbow clique of order $r$ and we are done. We assume then that there exists a vertex $v$ and a color, say 1, such that the neighborhood of $v$ in color 1 satisfies $|N_1(v)| \ge \left(\frac{c \cdot r^3}{\log r}\right)^{o-3}$.
		
		By the induction hypothesis applied to the restriction of $\chi$ to $N_1(v)$, the coloring $\chi$ contains either a rainbow $K_r$, or an orderable $K_{o-1}$. If the second case happens, then, by appending $v$ to the orderable $K_{o-1}$, we obtain an orderable $K_o$, as desired.
	\end{proof}
	
	\section{Proof of Theorem~\ref{thm:main2}~(\ref{ER(m,4,r)})} \label{sec:ER(4,4,r)}
	
	In the proof of Theorem~\ref{thm:main2}~(\ref{ER(m,4,r)}) we will apply Lemma~\ref{lem:ajmp} and analyze the coloring restricted to the neighborhood of a vertex $v$. We start with the following claims.
	
	\begin{claim} \label{claim1}
		Let $m\ge 3$ and $\chi: E(K_n) \to \w$ be a coloring. Assume there are vertices $u,v$ and distinct colors $i\neq j$ such that $u \in N_i(v)$ and $|N_j(u) \cap N_i(v)| \ge R(m-1)$. Then there exists in $\{u,v\} \cup (N_j(u) \cap N_i(v))$ either
		\begin{itemize}
			\item a lexical $K_4$ containing $u$ and $v$;
			\item a monochromatic $K_m$ containing either $u$ or $v$.
		\end{itemize}
	\end{claim}    
	
	\begin{claimproof}
		Any edge of color different from $i$ and $j$ in $N_j(u) \cap N_i(v)$ leads to a lexical $K_4$ (together with $u$ and $v$). If all edges in $N_j(u) \cap N_i(v)$ have color $i$ or $j$, then as $|N_j(u) \cap N_i(v)| \ge R(m-1)$, there is a monochromatic copy of $K_{m-1}$ in $N_j(u) \cap N_i(v)$. By adding either $u$ or $v$ we create a monochromatic copy of $K_m$.
	\end{claimproof}
	
	\begin{claim} \label{claim2}
		Let $\chi: E(K_n) \to \w$ be a coloring. Assume there are vertices $u,v$ such that $\chi(uv)=i$ and $|N_i(u) \cap N_i(v)| \ge \ER(4,\ell-1,r)$. Then there exists in $\{u,v\} \cup (N_i(u) \cap N_i(v))$ either a monochromatic $K_4$, a rainbow $K_r$, or a lexical $K_\ell$.
	\end{claim}    
	
	\begin{claimproof}
		Let $V \coloneqq N_i(u) \cap N_i(v)$. 
		Since $|V| \ge \ER(4,\ell-1,r)$, to avoid rainbow $K_r$ and monochromatic $K_4$ inside $V$, we must have a lexical $K_{\ell-1}$ inside $V$. 
		Note that if there is any edge $xy$ in color $i$ inside $V$, then $\{u,v,x,y\}$ induces a monochromatic $K_4$ in color $i$. We may assume then that there is no edge in color $i$ inside $V$. Thus, any lexical $K_{\ell-1}$ inside $V$ induces a lexical $K_\ell$ together with either $u$ or $v$. 
	\end{claimproof}
	
	\vspace{2mm}
	Similar to the proof of Lemma~\ref{lem:ajmp} in~\cite{ajmp}, we will show that given the upper bounds on color degrees given by Claims~\ref{claim1} and~\ref{claim2}, with positive probability a random subset of size $\Theta(r)$ contains very few color repetitions. By deleting a vertex from each edge with a repeated color, we obtain a rainbow $K_r$.
	
	\begin{proof}[Proof of Theorem~\ref{thm:main2}~(\ref{ER(m,4,r)})]
		First, we note that we can assume $r$ is sufficiently large. Indeed, by changing the constant if needed the result follows for small values of $r$. Let $c>0$ be a constant larger than the constant given by Lemma~\ref{lem:ajmp}. Note that $\ER(4,3,r) \le 4c \cdot r^3/\log r$.
		
		Let $n:= 60 r \cdot \left(c \cdot \frac{r^3}{\log r}\right)^{2}$ and $\chi: E(K_n) \to \w$ be a coloring. We will show that $\chi$ contains either a monochromatic $K_4$, lexical $K_4$, or rainbow $K_r$.
		
		If $\chi$ is $60r \cdot \left(c \cdot \frac{r^3}{\log r}\right)$-good, then, by Lemma~\ref{lem:ajmp}, there exists a rainbow clique of order $r$ under $\chi$.
		Let us assume then that there is a vertex $v$ and a color, say $1$, such that neighborhood of $v$ in color 1 has size $|N_1(v)| \ge 60r \cdot \left(c \cdot \frac{r^3}{\log r}\right)$.
		
		Let $V \coloneqq N_1(v)$.
		By Claims~\ref{claim1} and~\ref{claim2}, we may assume that every vertex $u \in V$ satisfies 
		\begin{align} 
			&|N_i(u) \cap V| < 6 \text{ for every } i \neq 1. \label{eq:1} \\
			&|N_1(u) \cap V| < 4c \cdot \left(\frac{r^3}{\log r}\right). \label{eq:2}
		\end{align}
		
		Let $N \coloneqq |V|$.
		For any subset $S \subset V$, let $X(S)$ be the number of monochromatic $K_{1,2}$ in $S$ in a color different from 1; $Y(S)$ be the number of monochromatic $2 K_{2}$ in $S$ in a color different from 1; and $Z(S)=e_1(S)$ be the number of edges with color 1 in $S$. Let $T \subset V$ with $|T|=3r$ be chosen uniformly at random. We will show that all random variables $X(T), Y(T), Z(T)$ are small in expectation. For a particular instance of such subset where $X, Y, Z$ are small, we will be able to find a rainbow copy of $K_r$.
		
		First, we bound the number of monochromatic structures in colors different from 1 inside $V$ using~\eqref{eq:1}. Note that there are at most $5N/2$ edges of each given color $i\neq 1$. Hence,
		$$
		X(V) \le N \cdot \frac{N}{5} \cdot\binom{5}{2} = 2 N^{2}
		\quad \text{and} \quad
		Y(V) \le \frac{1}{2}\binom{N}{2} \cdot \frac{5N}{2} \le \frac{5 N^3}{8} .
		$$
		For any fixed copy $H$ of a monochromatic $K_{1,2}$ in $V$ in a color different from 1, the probability that $T$ contains all three vertices of $H$ is at most $(3r/N)^{3}$. 
		Similarly, for any fixed copy $H$ of a monochromatic $2 K_{2}$ in $V$ in a color different from 1, the probability that $T$ contains all four vertices of $H$ is at most $(3r/N)^{4}$.
		Therefore, by linearity of expectation,
		$$ \E[X(T)+Y(T)] \le 
		2  N^{2}\left(\frac{3r}{N}\right)^{3} + 
		\frac{5 N^3}{8}  \left(\frac{3r}{N}\right)^{4}
		= \frac{ 432 r^{3}+ 405 r^{4} }{8N} 
		< \frac{r}{3}. $$
		By Markov’s Inequality, we conclude
		$$\P(X(T)+Y(T)\ge r) \le \frac{1}{3} .$$ 
		
		From~\eqref{eq:2}, the number of edges in color 1 inside $V$, $e_1(V)$, is at most
		$$ e_1(V) <  2c \cdot \left(\frac{r^3}{\log r}\right) N \le \frac{N^2}{30r} .$$
		For each edge of color 1 in $V$, the probability that $T$ contains both of its endpoints is at most $(3r/N)^2$, then
		\begin{align*}
			\E[Z(T)] < e_1(V) \cdot \frac{(3r)^2}{N^2} 
			< \frac{N^2}{30r} \cdot \frac{9r^2}{N^2} < \frac{r}{3} 
		\end{align*}
		and, by Markov's inequality, we have 
		$$\P \left( Z(T) \ge r \right) \le \frac{1}{3} .$$
		
		Therefore, with positive probability, there exists a $3r$-subset $T \subset V$ that contains at most $r$ edges of color 1, and at most $r$ monochromatic $2 K_{2}$ and $K_{1,2}$ in colors different from 1. By deleting one vertex from each edge of color $1$, and one vertex from each monochromatic $2 K_{2}$ and $K_{1,2}$ from $T$, we obtain a subset $T' \subset T$ of order at least $r$ containing no monochromatic $2 K_{2}$ and $K_{1,2}$. Since such $T'$ induces a rainbow clique, the result follows.
	\end{proof}
	
	\section{Lower bound for $\ER(3,\ell,r)$}\label{sec:lower}
	
	In this section we will show the following.
	
	\begin{claim}
		For every integers $\ell \ge 3$ and $r \ge 3$, we have that
		$$\ER(3,\ell+1,r) \ge (\ER(3,\ell,r)-1) \cdot (\ER(3,3,r)-1) + 1.$$
	\end{claim}
	
	The case $\ell = 3$ in Theorem~\ref{thm:lower}, namely that $\ER(3,3,r) \ge \Omega( r^3 / \log r)$, follows from~\cite{babai}. Theorem~\ref{thm:lower} then follows immediately from the claim. Indeed, by induction on $\ell \ge 3$, the bound $\ER(3,\ell,r)-1 \ge (c r^3 / \log r)^{\ell-2}$ implies 
	$$ \ER(3,\ell+1,r)-1 \ge \left( \frac{c r^3}{\log r} \right)^{\ell-2} \cdot \left( \frac{c r^3}{\log r} \right) = \left( \frac{c r^3}{\log r} \right)^{\ell-1} .$$
	
	\begin{claimproof}
		Let $N_\ell \coloneqq \ER(3,\ell,r)-1$. By the definition of the number $\ER(m,\ell,r)$, there exists a coloring $\chi_\ell: \binom{[N_\ell]}{2} \to \w$ avoiding monochromatic $K_3$, lexical $K_\ell$, and rainbow $K_r$. In what follows we consider colorings $\chi_3$ and $\chi_\ell$ that use a disjoint set of colors\footnote{We highlight that when $\ell = 3$ we consider two distinct colorings of $\binom{[N_3]}{2}$ using a disjoint set of colors.}.
		
		We will define a coloring $\chi: \binom{[N_{\ell} \cdot N_3]}{2} \to \w$ as follows. 
		Partition the vertices $[ N_{\ell} \cdot N_3 ]$ into sets $V_1, \dots , V_{N_3}$ of equal size $|V_i|=N_\ell$ by setting 
		$$V_i \coloneqq [ (i-1) N_\ell + 1, i N_\ell ] \text{ for } 1 \le i \le N_3.$$
		For vertices $u \in V_i$ and $v \in V_j$ when $i\neq j$ we set 
		$$\chi(uv) \coloneqq \chi_3(ij).$$ 
		For every $i\in [N_3]$ and vertices $u, v \in V_i$ we write $u = (i-1) \cdot N_\ell + u'$, $v = (i-1) \cdot N_\ell + v'$, and set 
		$$\chi(uv) \coloneqq \chi_\ell(u'v') .$$
		
		We claim $\chi$ contains no monochromatic $K_3$, lexical $K_{\ell+1}$, or rainbow $K_r$.
		First, $\chi$ contains no monochromatic $K_3$. Indeed,
		if $a,b,c \in V_i$ then $abc$ is not monochromatic since $\chi_\ell$ contains no monochromatic $K_3$. If $a,b \in V_i$ and $c\in V_j$ for $i\neq j$ then $\chi(ab) \neq \chi(ac)$ since the colors from $\chi_3$ and $\chi_\ell$ are disjoint. Finally, if $a\in V_i$, $b\in V_j$, $c\in V_k$ for distinct $i,j,k$, then $abc$ is not monochromatic since $\chi_3$ does not contain monochromatic $K_3$.
		
		Now, $\chi$ contains no lexical $K_{\ell+1}$. Indeed,
		let $v_1, v_2, \dots, v_{\ell+1}$ be a potential lexical clique, meaning $\chi(v_iv_j)=\chi(v_rv_s)$ for $i<j$ and $r<s$ if and only if $i=r$. If $v_1, v_2 \in V_i$ for some $i$ then $\chi(v_1v_2)=\chi(v_1u)$ only for vertices $u\in V_i$, we conclude $v_1, v_2, \dots, v_{\ell+1} \in V_i$. As $\chi_\ell$ contains no lexical $K_\ell$, we conclude $v_1, \dots, v_{\ell+1}$ cannot induce a lexical clique. Assume then $v_1\in V_i$ and $v_2 \in V_j$ for $i\neq j$. As $\chi_3$ is a proper coloring, $\chi(v_1v_2)=\chi(v_1u)$ only for vertices $u\in V_j$, we conclude $v_2, v_3, \dots, v_{\ell+1} \in V_j$. Again, as $\chi_\ell$ contains no lexical $K_\ell$, we conclude $v_1, \dots, v_{\ell+1}$ cannot induce a lexical clique.
		
		Finally, $\chi$ contains no rainbow $K_r$. Indeed,
		if $a,b \in V_i$ and $c\in V_j$ for $i\neq j$ then $\chi(ac)=\chi(bc)$. The only two possibilities for a rainbow clique in $\chi$ are 1) if all vertices belong to same set $V_i$ for some $i$, or 2) if they all belong to different sets. As $\chi_3$ and $\chi_\ell$ does not contain rainbow clique of order $r$, we conclude $\chi$ does not contain rainbow clique of order $r$ as well.   
	\end{claimproof}
	
	\section{Concluding remarks}
	
	In this paper we obtained upper bounds for the numbers $\CR(o,r)$ and $\ER(m,\ell,r)$. We obtained the correct asymptotics for $\CR(o,r) = \Theta(r^3/\log r)^{o-2}$, $\ER(m,3,r) = \Theta(r^3/\log r)$ and $\ER(3,\ell,r) =\Theta(r^3/\log r)^{\ell-2}$. However, there is still a gap between the known lower bounds and our upper bounds for $\ER(m,\ell,r)$ for general $m$ and $\ell$. We expect the known lower bounds to describe the correct asymptotic in this general case as well.
	
	\begin{conjecture} \label{conj:ER(m,ell,r)}
		For every $m, \ell \ge 4$, there exists a constant $c=c(m,\ell)>0$ such that, for every $r\ge 3$,
		$$ \ER(m,\ell,r) \le c \cdot \left(\frac{r^3}{\log r}\right)^{\ell-2}.$$
	\end{conjecture}
	
	The smallest case where we do not obtain the correct asymptotic and potential easiest case to deal with is when we forbid monochromatic and lexical $K_4$.
	
	\begin{conjecture} \label{conj:ER(4,4,r)}
		There exists a constant $c>0$ such that, for every $r\ge 3$,
		$$ \ER(4,4,r) \le c \cdot \left(\frac{r^3}{\log r}\right)^{2}.$$
	\end{conjecture}
	
	Furthermore, when we only forbid a lexical $K_4$, we can show that $\ER(m,4,r) = O( r^{m+3 +o(1)})$ by a proof analogous to the one for $\ER(4,4,r)\le c \cdot r^{7}/(\log r)^2$. But we expect the correct behavior to be a much lower order of magnitude in the exponent of $r$.
	
	\begin{conjecture} \label{conj:ER(m,4,r)}
		For every $m \ge 4$, there exists a constant $c=c(m)>0$ such that, for every $r\ge 3$,
		$$ \ER(m,4,r) \le c \cdot \left(\frac{r^3}{\log r}\right)^{2}.$$
	\end{conjecture}
	
	\section*{Acknowledgments}
	
	The research in this paper was initiated when the first author participated in the 1st Brazilian School of Combinatorics in Juquehy, S\~ao Paulo, Brazil. The first author is grateful for the organizers of the event and many of its participants for helpful discussions that led to this work. Both authors are also grateful to our adviser J\'ozsef Balogh for helpful discussions and carefully reading a previous version of this manuscript.

\end{document}